\documentclass[11pt]{amsart}

\usepackage{amscd,amsmath,latexsym,amsthm,amsfonts,amssymb,graphicx,geometry}

\usepackage[T1]{fontenc}
\usepackage[utf8]{inputenc}
\usepackage[english]{babel}
\usepackage{lmodern}
\usepackage{microtype}
\usepackage{amsmath,amssymb,amsthm,mathtools}
\usepackage{enumitem}
\usepackage{hyperref}
\usepackage[nameinlink,noabbrev]{cleveref}

\allowdisplaybreaks

\newtheorem{theorem}{Theorem}[section]
\newtheorem{proposition}[theorem]{Proposition}
\newtheorem{lemma}[theorem]{Lemma}
\newtheorem{corollary}[theorem]{Corollary}

\theoremstyle{definition}
\newtheorem{definition}[theorem]{Definition}

\theoremstyle{remark}
\newtheorem{remark}[theorem]{Remark}

\newcommand{\K}{\mathbb{K}}
\newcommand{\N}{\mathbb{N}}
\newcommand{\Lop}{\mathcal{L}}
\newcommand{\Kop}{\mathcal{K}}
\newcommand{\SSop}{\mathcal{SS}}
\newcommand{\Cal}{\operatorname{Cal}}
\newcommand{\Id}{I}
\newcommand{\ran}{\operatorname{ran}}
\newcommand{\codim}{\operatorname{codim}}
\newcommand{\ind}{\operatorname{ind}}
\newcommand{\Dom}{D}

\begin{document}
\title[The Lotz Property in Weakly Compactly Generated Spaces]{The Lotz Property in Weakly Compactly Generated Spaces}

\author[Araújo]{Gustavo Araújo}
\address[G. Araújo]{Departamento de Matem\'{a}tica \newline\indent
	Universidade Estadual da Para\'{i}ba \newline\indent
	Campina Grande - PB \newline\indent
	58.429-500 (Brazil)}
\email{gustavoaraujo@servidor.uepb.edu.br}

\author[Azevedo]{Mateus Azevedo}
\address[M. Azevedo]{Unidade Acad\^emica de Matemática \newline\indent
	Universidade Federal de Campina Grande \newline\indent
	Campina Grande - PB, Brazil \newline\indent
	58.429-900 (Brazil)}
\email{mateusimjs19@gmail.com}

\author[Bezerra]{Flank Bezerra}
\address[F. Bezerra]{Departamento de Matem\'atica \newline\indent
	Universidade Federal da Para\'iba \newline\indent
	Jo\~ao Pessoa - PB \newline\indent
	58.051-900 (Brazil)}
\email{flank@mat.ufpb.br}

\author[Barbosa]{Anderson Barbosa}
\address[A. Barbosa]{Departamento de Ciências Exatas e Tecnologia da Informação \newline\indent
	Universidade Federal Rural do Semi-\'Arido \newline\indent
	Rio Grande do Norte - RN \newline\indent
	59.515-000 (Brazil)}
\email{anderson.barbosa@ufersa.edu.br}



\thanks{G. Araújo was supported by Grant 2457/2026 from the Para\'iba State Research Foundation (FAPESq). F. Bezerra was partially supported by CNPq/Brasil through Grants 303039/2021-3 and 301466/2026-2.}

\begin{abstract}
	We study whether every infinite-dimensional weakly compactly generated (WCG) Banach space fails the Lotz property.  The answer is affirmative in the nonseparable case: a classical decomposition theorem for WCG spaces provides a Schauder decomposition, and an explicit diagonal $C_0$-semigroup has an unbounded generator.  The unrestricted assertion is false.  Our main abstract result is a Calkin-algebra criterion: if $X$ fails the bounded compact approximation property and, for some fixed $m$, every noninvertible element of $\Cal(X)=\Lop(X)/\Kop(X)$ has $m$th power zero, then $X$ has the Lotz property. In particular, this applies when every operator is scalar-plus-strictly-singular and one fixed power of every strictly singular operator is compact. Combining this criterion with the Maurey--Pisier--Szankowski extraction theorem and constructions of Argyros and Motakis yields separable hereditarily indecomposable Lotz spaces over both scalar fields: reflexive examples, and, over the complex field, an example containing no infinite-dimensional reflexive subspace.  Moreover, the ambient Argyros--Motakis spaces considered here are Lotz-saturated, although they themselves are non-Lotz.  Thus every nonseparable WCG space is non-Lotz, while the separable WCG class contains both Lotz and non-Lotz spaces.
\end{abstract}

\maketitle

\tableofcontents

\section{Introduction}

Let $X$ be a Banach space over $\K$, where $\K=\mathbb R$ or $\mathbb C$.
A family $(T(t))_{t\geq0}\subset\Lop(X)$ is a \emph{$C_0$-semigroup} if
\[
T(0)=\Id_X,\qquad T(t+s)=T(t)T(s)\quad(t,s\geq0),
\]
and
\[
\lim_{t\downarrow0}\|T(t)x-x\|=0\qquad(x\in X).
\]
Its \emph{infinitesimal generator} is the operator
\[
Ax=\lim_{t\downarrow0}\frac{T(t)x-x}{t},
\]
defined on the set $\Dom(A)$ of all $x\in X$ for which this limit exists.
The semigroup is \emph{uniformly continuous} if
\[
\lim_{t\downarrow0}\|T(t)-\Id_X\|=0.
\]
We say that $X$ has the \emph{Lotz property} if every $C_0$-semigroup on
$X$ is uniformly continuous.  Equivalently, every infinitesimal generator of a
$C_0$-semigroup on $X$ is bounded.  Lotz proved that every Grothendieck space
with the Dunford--Pettis property has this property
\cite{Lotz1985}; see also the partial converse for Banach lattices obtained by
van Neerven \cite{vanNeerven1992} and the modern account in
\cite{Budde2024}.

The question that motivates this paper is the following:
\begin{equation}\label{eq:main-question}
\text{Must every infinite-dimensional WCG Banach space be non-Lotz?}
\end{equation}
Here \emph{non-Lotz} means that at least one $C_0$-semigroup on the space is
not uniformly continuous.  The universal assertion is true for nonseparable
WCG spaces but false in the separable class.  This is not a dichotomy of
individual spaces: separable WCG spaces include both Lotz and non-Lotz
examples.

The positive answer of this question  comes from Schauder decompositions.  The classical
structure theory of weakly compactly generated spaces
\cite{AmirLindenstrauss1968,Lindenstrauss1971} implies that every
nonseparable WCG space admits such a decomposition.  We give all details of the construction of the associated multiplier: if $(X_n)_{n\in\mathbb{N}}$ is a Schauder
decomposition, then
\[
T(t)\Big(\sum_{n=1}^\infty x_n\Big)
=\sum_{n=1}^\infty e^{-nt}x_n
\qquad(t\geq0)
\]
defines a bounded $C_0$-semigroup whose infinitesimal generator equals $-n\Id$ on $X_n$. The infinitesimal generator is therefore unbounded, so the space is non-Lotz.

The negative answer to this question  is established by the construction presented in this paper.  We produce a
separable, reflexive, infinite-dimensional Lotz space $X_L$.  Every separable
Banach space is WCG, so $X_L$ is a counterexample to the unrestricted form
of \eqref{eq:main-question}.  Combining both conclusions gives the main
answer.

\begin{theorem}\label{thm:answer-intro}
	\begin{enumerate}[label=\textup{(\roman*)}]
		\item Every nonseparable WCG Banach space is non-Lotz.
		\item There exists a separable, reflexive, hereditarily indecomposable,
		infinite-dimensional real WCG Banach space with the Lotz property.
		\item There exists a separable, hereditarily indecomposable, complex WCG
		Banach space with the Lotz property and with no infinite-dimensional
		reflexive subspace.
	\end{enumerate}
	Consequently, not every infinite-dimensional WCG space is non-Lotz.
\end{theorem}

The route to these results connects the Lotz property with the Calkin algebra
and approximation properties.  Write $\Lop(X)$, $\Kop(X)$ and $\SSop(X)$
for, respectively, the bounded, compact and strictly singular operators on
$X$, and put
\[
\Cal(X):=\Lop(X)/\Kop(X).
\]
Our abstract criterion is the following.

\begin{theorem}\label{thm:criterion-intro}
	Let $X$ fail the bounded compact approximation property.  Suppose that there
	is $m\in\N$ such that every noninvertible $a\in\Cal(X)$ satisfies $a^m=0$.
	Then $X$ has the Lotz property.
\end{theorem}

The mechanism is short.  If an unbounded infinitesimal generator $A$ existed, then every
sufficiently far right resolvent $R(\lambda,A)$ would be non-Fredholm.  By
Atkinson's theorem its image in the Calkin algebra would be noninvertible and
hence nilpotent of index at most $m$.  Consequently the Yosida approximants
\[
J_\lambda=\lambda R(\lambda,A),\ R(\lambda,A):=(\lambda I_X-A)^{-1}
\]
would have compact $m$th powers.  These powers remain uniformly bounded and
converge strongly to the identity, giving the bounded compact approximation
property, a contradiction.

A convenient consequence applies to spaces with few operators.  It suffices
that
\begin{equation}\label{eq:scalar-plus-ss-intro}
	\Lop(X)=\K\Id_X+\SSop(X)
\end{equation}
and, for one fixed $m$, that $S^m$ be compact for every
$S\in\SSop(X)$.  Notice that this power condition is weaker than requiring
every product of $m$ possibly different strictly singular operators to be
compact.

The criterion applies hereditarily inside several spaces of Argyros and
Motakis.  Every infinite-dimensional subspace of
$\mathfrak X_{\mathrm{ISP}}$ contains an $\ell_1$ spreading model, has the
scalar-plus-strictly-singular property, and has compact triple products of
strictly singular operators \cite{ArgyrosMotakis2014}.  The spreading model
rules out every nontrivial Rademacher type.  The
Maurey--Pisier--Szankowski extraction theorem then supplies a further
subspace without the compact approximation property.  The same method
applies, with compact squares, to the complex versions of the spaces from the
dual construction in \cite{ArgyrosMotakisDual}.

\begin{theorem}\label{thm:main-intro}
	The following separable hereditarily indecomposable Lotz spaces exist:
	\begin{enumerate}[label=\textup{(\roman*)}]
		\item a real reflexive subspace of $\mathfrak X_{\mathrm{ISP}}$;
		\item complex reflexive examples coming from the well-founded-tree
		spaces of Argyros and Motakis;
		\item a complex example, contained in their universal-tree space, with
		no infinite-dimensional reflexive subspace.
	\end{enumerate}
	All these spaces are WCG.  In fact, the corresponding ambient
	Argyros--Motakis spaces are Lotz-saturated: every infinite-dimensional closed
	subspace contains an infinite-dimensional closed Lotz subspace.
\end{theorem}

The distinction between the approximation property and the compact
approximation property is essential here: the latter is strictly weaker, as
shown by Willis \cite{Willis1992}.  Thus it would not be enough merely to
extract a subspace without the approximation property.

The classical sufficient condition of Lotz produces no
infinite-dimensional separable examples.  Recall that $X$ is a
\emph{Grothendieck space} if every weak-star convergent sequence in $X^*$ is
weakly convergent, and $X$ has the \emph{Dunford--Pettis property} if every
weakly compact operator from $X$ into an arbitrary Banach space is completely
continuous, that is, it sends weakly convergent sequences to norm-convergent
sequences.  A separable Grothendieck space is reflexive: the dual unit ball is
weak-star compact and metrizable, so every sequence in it has a weak-star
convergent subsequence; the Grothendieck property makes that subsequence
weakly convergent, and the Eberlein--\v Smulian theorem yields weak compactness
of the dual unit ball.  If, in addition, $X$ has the Dunford--Pettis property,
then the weakly compact operator $\Id_X$ is completely continuous.  Every
bounded sequence consequently has a norm-convergent subsequence, so the unit
ball is norm compact and $X$ is finite-dimensional. Thus the standard
Grothendieck--Dunford--Pettis class of Lotz spaces is necessarily
nonseparable in infinite dimension.  Against this background, the above examples
give a new mechanism for producing separable Lotz spaces.  The author
is not aware of a previously recorded infinite-dimensional separable
example; see
\cite{GonzalezKania2021,Budde2024} for background on the two classical
properties and on Lotz's theorem.

The remainder of this paper is structured as follows. Section \ref{SecWCGSpacSchuderDec} introduces WCG spaces and Schauder decompositions and discusses their  relationship with the Lotz property. Section \ref{SecPrelimin} presents fundamental conceptions and definitions of the $C_0-$semigroup theory, compact approximation properties, and relevant notions from the Fredholm theory. Section \ref{SecCalkinAlg} presents a  Calkin criterion on the Lotz property. Section \ref{SecArgyros-Motakis} states and proves the results about Argyros–Motakis space. Concluding the proofs of our main results are finalized, and remarks are discussed in Section \ref{SecWCG}.

\section{WCG spaces and Schauder decompositions}\label{SecWCGSpacSchuderDec}

\begin{definition}
	A Banach space $X$ is \emph{weakly compactly generated}, abbreviated WCG,
	if there is a weakly compact set $K\subset X$ such that
	\[
	\overline{\operatorname{span}}K=X.
	\]
	The closure is taken in the norm topology.  This class was introduced and
	systematically studied by Amir and Lindenstrauss
	\cite{AmirLindenstrauss1968}.
\end{definition}

\begin{lemma}\label{lem:separable-WCG}
	Every separable Banach space is WCG.
\end{lemma}

\begin{proof}
	Choose a norm-dense sequence $(u_n)_{n\in\mathbb{N}}$ in the closed unit ball of
	$X$ and put
	\[
	K:=\{0\}\cup\{2^{-n}u_n:n\in\N\}.
	\]
	The sequence $2^{-n}u_n$ converges to zero in norm.  Hence $K$ is norm
	compact and therefore weakly compact.  Since multiplication by a nonzero
	scalar does not change linear span,
	\[
	\overline{\operatorname{span}}K
	=\overline{\operatorname{span}}\{u_n:n\in\N\}=X.
	\]
	Thus $X$ is WCG.
\end{proof}

\begin{definition}\label{def:Schauder-decomposition}
	A sequence $(X_n)_{n\in\mathbb{N}}$ of nonzero closed subspaces of a Banach space
	$X$ is a \emph{Schauder decomposition} of $X$ if every $x\in X$ has a
	unique representation
	\[
	x=\sum_{n=1}^\infty x_n,
	\qquad x_n\in X_n,
	\]
	where the series converges in norm.  The coordinate maps $Q_nx=x_n$ and
	the partial-sum projections
	\[
	P_Nx:=\sum_{n=1}^NQ_nx
	\]
	are bounded.  Since $P_Nx\to x$ for every $x\in X$, the uniform boundedness
	principle gives the finite decomposition constant
	\begin{equation}\label{eq:decomposition-constant}
		K_D:=\sup_{N\in\mathbb{N}}\|P_N\|<\infty.
	\end{equation}
	A Schauder basis is the special case in which every $X_n$ is
	one-dimensional.
\end{definition}

\begin{proposition}\label{prop:Schauder-non-Lotz}
	Every Banach space admitting a Schauder decomposition is non-Lotz.
\end{proposition}

\begin{proof}
	Let $(X_n)_{n\in\mathbb{N}}$ be a Schauder decomposition and use the notation of
	\cref{def:Schauder-decomposition}.  For $t>0$ and $N\in\N$, set
	\[
	T_N(t):=\sum_{n=1}^Ne^{-nt}Q_n.
	\]
	Abel summation gives the exact operator identity
	\begin{equation}\label{eq:Abel-multiplier}
		\sum_{n=1}^Na_nQ_n
		=a_NP_N+\sum_{n=1}^{N-1}(a_n-a_{n+1})P_n
	\end{equation}
	for arbitrary scalars $a_1,\ldots,a_N$.  With $a_n=e^{-nt}$, the sequence
	$(a_n)$ is positive and decreasing, and therefore
	\begin{align}
		\|T_N(t)\|
		&\leq K_D\left(a_N+\sum_{n=1}^{N-1}(a_n-a_{n+1})\right) \notag\\
		&=K_Da_1=K_De^{-t}.                         \label{eq:diagonal-bound}
	\end{align}
	Moreover, if $M>N$, another application of Abel summation yields
	\begin{align*}
		\sum_{n=N+1}^Ma_nQ_n
		&=a_MP_M-a_{N+1}P_N
		+\sum_{n=N+1}^{M-1}(a_n-a_{n+1})P_n,
	\end{align*}
	and hence
	\begin{align}
		\left\|\sum_{n=N+1}^Me^{-nt}Q_n\right\|
		&\leq K_D\left(a_M+a_{N+1}
		+\sum_{n=N+1}^{M-1}(a_n-a_{n+1})\right)\notag\\
		&=2K_Da_{N+1}\\
		&=2K_De^{-(N+1)t}.              \label{eq:diagonal-tail}
	\end{align}
	Thus $(T_N(t))_N$ is Cauchy in operator norm.  Define
	\begin{equation}\label{eq:diagonal-semigroup}
		T(0):=\Id_X,
		\qquad
		T(t):=\sum_{n=1}^\infty e^{-nt}Q_n\quad(t>0).
	\end{equation}
	Passing to the limit in \eqref{eq:diagonal-bound} gives
	\begin{equation}\label{eq:diagonal-uniform-bound}
		\|T(t)\|\leq K_De^{-t}\leq K_D
		\qquad(t>0).
	\end{equation}
	
	For $x=\sum_nx_n$, the continuity of each coordinate projection gives
	$Q_nT(t)x=e^{-nt}x_n$.  Consequently,
	\[
	T(t)T(s)x
	=\sum_{n=1}^\infty e^{-nt}e^{-ns}x_n
	=T(t+s)x,
	\]
	so $(T(t))_{t\geq0}$ has the semigroup property.
	
	We verify strong continuity without suppressing the tail estimate.  Given
	$x\in X$ and $\varepsilon>0$, choose $N$ so large that
	\[
	(K_D+1)\|x-P_Nx\|<\frac{\varepsilon}{2}.
	\]
	For $t>0$, \eqref{eq:diagonal-uniform-bound} implies
	\begin{align*}
		\|T(t)x-x\|
		&\leq \|T(t)(x-P_Nx)\|+\|x-P_Nx\|
		+\|T(t)P_Nx-P_Nx\|\\
		&\leq (K_D+1)\|x-P_Nx\|
		+\sum_{n=1}^N|e^{-nt}-1|\,\|Q_nx\|.
	\end{align*}
	The finite sum tends to zero as $t\downarrow0$, proving
	$T(t)x\to x$.  The same finite-tail argument, with
	$|e^{-nt}-e^{-ns}|$ in place of $|e^{-nt}-1|$, proves strong continuity at
	every $s>0$.  Therefore $(T(t))_{t\geq0}$ is a $C_0$-semigroup.
	
	Let $A$ be its infinitesimal generator.  If $u\in X_n$, then
	\[
	T(t)u=e^{-nt}u
	\]
	and hence
	\begin{equation}\label{eq:diagonal-generator}
		Au=\lim_{t\downarrow0}\frac{e^{-nt}-1}{t}u=-nu.
	\end{equation}
	In fact, the whole infinitesimal generator is described by
	\begin{equation}\label{eq:diagonal-domain}
		\Dom(A)=\left\{x\in X:\sum_{n=1}^\infty nQ_nx
		\text{ converges in }X\right\},
		\qquad
		Ax=-\sum_{n=1}^\infty nQ_nx.
	\end{equation}
	To verify this, first suppose that $Ax$ exists.  Since $Q_n$ is bounded,
	\[
	Q_nAx
	=\lim_{t\downarrow0}Q_n\frac{T(t)x-x}{t}
	=\lim_{t\downarrow0}\frac{e^{-nt}-1}{t}Q_nx
	=-nQ_nx.
	\]
	Therefore
	$P_NAx=-\sum_{n=1}^NnQ_nx\to Ax$, proving convergence of the series in
	\eqref{eq:diagonal-domain}.  Conversely, suppose
	$y:=\sum_{n=1}^\infty nQ_nx$ exists.  Put
	\[
	b_n(t):=\frac{1-e^{-nt}}{nt}\qquad(t>0).
	\]
	For every $t>0$, $0<b_n(t)\leq1$, the sequence $(b_n(t))_n$ is decreasing,
	because the derivative of $f(s)=(1-e^{-s})/s$ is
	\[
	f'(s)=\frac{(s+1)e^{-s}-1}{s^2}\leq0
	\quad(s>0)
	\]
	by $e^s\geq1+s$; moreover, $b_n(t)\to1$ as $t\downarrow0$ for each fixed
	$n$.  Abel summation,
	now applied to the Schauder expansion
	$y=\sum_{n=1}^\infty nQ_nx$, shows that the diagonal multipliers
	$\sum_{n=1}^\infty b_n(t)Q_n$ are uniformly bounded by $K_D$.  The same finite-tail
	argument used above consequently gives
	\[
	\sum_{n=1}^\infty b_n(t)nQ_nx\longrightarrow y
	\qquad(t\downarrow0).
	\]
	Since
	\[
	\frac{T(t)x-x}{t}
	=-\sum_{n=1}^\infty b_n(t)nQ_nx,
	\]
	we obtain $x\in\Dom(A)$ and $Ax=-y$, proving
	\eqref{eq:diagonal-domain}.
	
	Choose $u_n\in X_n$ with $\|u_n\|=1$.  Equation
	\eqref{eq:diagonal-generator} gives
	\[
	\|Au_n\|=n\longrightarrow\infty.
	\]
	Thus $A$ is unbounded.  Since a $C_0$-semigroup is uniformly continuous if
	and only if its infinitesimal generator is bounded
	\cite[Theorem~I.3.7]{EngelNagel2000}, the semigroup in
	\eqref{eq:diagonal-semigroup} is not uniformly continuous.  Hence $X$ is
	non-Lotz.
\end{proof}

We now use the classical WCG decomposition theorem.  In the language needed
here, it is a consequence of the projectional resolution of the identity
furnished by the Amir--Lindenstrauss structure theorem
\cite{AmirLindenstrauss1968}; see Lindenstrauss's decomposition note
\cite{Lindenstrauss1971}.

A \emph{bounded projectional resolution of the identity} (PRI), indexed by
an ordinal $\kappa$, is a family of bounded projections
$(P_\alpha)_{0\leq\alpha\leq\kappa}$ such that
\[
P_0=0,\qquad P_\kappa=\Id_X,
\qquad \sup_{0\leq\alpha\leq\kappa}\|P_\alpha\|<\infty,
\]
and
\[
P_\alpha P_\beta=P_\beta P_\alpha
=P_{\min\{\alpha,\beta\}}
\qquad(0\leq\alpha,\beta\leq\kappa).
\]
At every limit ordinal $\gamma\leq\kappa$, it is also required that
\[
P_\gamma X
=\overline{\bigcup_{\alpha<\gamma}P_\alpha X}.
\]
The \emph{density character} $\operatorname{dens}(Y)$ of a topological space
$Y$ is the least cardinality of a dense subset of $Y$.  For a WCG Banach
space $X$, the classical structure theorem supplies a PRI indexed by the
initial ordinal of $\operatorname{dens}(X)$ and satisfying
\[
\operatorname{dens}(P_\alpha X)
\leq\max\{\aleph_0,|\alpha|\}
\qquad(\alpha<\kappa).
\]
In particular, when $X$ is nonseparable, this resolution contains infinitely
many distinct projections.

\begin{theorem}[WCG decomposition theorem]
	\label{thm:nonseparable-WCG-decomposition}
	Every nonseparable WCG Banach space admits a Schauder decomposition.
\end{theorem}

\begin{proof}
	We spell out the countable extraction from the classical PRI.  Choose a
	strictly increasing sequence of indices
	\[
	\alpha_1<\alpha_2<\cdots
	\]
	such that the corresponding projections are distinct, and put
	$\alpha=\sup_n\alpha_n$.  Such a sequence exists because the resolution is
	continuous and $X$ is nonseparable.  Let
	\[
	E_1=P_{\alpha_1}X,
	\qquad
	E_n=(P_{\alpha_n}-P_{\alpha_{n-1}})X\quad(n\geq2),
	\]
	discarding an initial zero range if necessary.  The $E_n$ are nonzero closed
	subspaces.  Continuity of the PRI gives
	\begin{equation}\label{eq:PRI-limit}
		P_{\alpha_n}x\longrightarrow P_\alpha x
		\qquad(x\in X).
	\end{equation}
	If $P_\alpha=\Id_X$, then $(E_n)_{n\in\mathbb{N}}$ is a Schauder decomposition: its
	$N$th partial-sum projection is $P_{\alpha_N}$, and
	\eqref{eq:PRI-limit} gives convergence to the identity.
	
	If $P_\alpha\neq\Id_X$, put
	\[
	E_0=(\Id_X-P_\alpha)X.
	\]
	Then $E_0$ is nonzero and
	\[
	X=E_0\oplus\overline{\operatorname{span}}\{E_n:n\in\mathbb{N}\}.
	\]
	For the ordered sequence $(E_0,E_1,E_2,\ldots)$, the partial-sum
	projections are
	\[
	S_N=\Id_X-P_\alpha+P_{\alpha_N}.
	\]
	They are uniformly bounded and, by \eqref{eq:PRI-limit}, satisfy
	$S_Nx\to x$ for every $x\in X$.  The commutation relation for the PRI makes
	the summands pairwise algebraically disjoint and gives uniqueness of the
	expansion.  Hence this sequence is again a Schauder decomposition.
\end{proof}

\begin{corollary}\label{cor:nonseparable-WCG-non-Lotz}
	Every nonseparable WCG Banach space is non-Lotz.
\end{corollary}

\begin{proof}
	Apply \cref{thm:nonseparable-WCG-decomposition} and then
	\cref{prop:Schauder-non-Lotz}.
\end{proof}

\section{Preliminaries}\label{SecPrelimin}

\subsection{Semigroups and their resolvents}

Let $(T(t))_{t\geq0}$ be a $C_0$-semigroup on $X$, with infinitesimal generator $A$.
As usual, $A$ is closed and densely defined.  Every $C_0$-semigroup satisfies
a growth estimate
\begin{equation}\label{eq:growth}
	\|T(t)\|\leq Me^{\omega t}
	\qquad(t\geq0)
\end{equation}
for some $M\geq1$ and $\omega\in\mathbb R$.  Its \emph{growth bound} is
\[
\omega_0(T)
:=
\inf\left\{
\nu\in\mathbb R:
\text{there is }M_\nu\geq1\text{ such that }
\|T(t)\|\leq M_\nu e^{\nu t}\text{ for every }t\geq0
\right\}.
\]

For a closed densely defined operator $A\colon\Dom(A)\subset X\to X$, its
\emph{resolvent set} is
\[
\rho(A)
:=
\left\{
\lambda\in\K:
\lambda\Id_X-A\colon\Dom(A)\to X
\text{ is bijective and has a bounded inverse}
\right\}.
\]
For $\lambda\in\rho(A)$, the operator
\[
R(\lambda,A):=(\lambda\Id_X-A)^{-1}
\]
is called the \emph{resolvent} of $A$ at $\lambda$.

Fix $M\geq1$ and $\omega\in\mathbb R$ satisfying
\eqref{eq:growth}.  If the real number $\lambda$ satisfies $\lambda>\omega$,
then $\lambda\in\rho(A)$ and
\begin{equation}\label{eq:laplace-resolvent}
	R(\lambda,A)x
	=\int_0^\infty e^{-\lambda t}T(t)x\,dt,
\end{equation}
where the integral is a Bochner integral.  In particular,
\begin{equation}\label{eq:resolvent-bound}
	\|R(\lambda,A)\|
	\leq\int_0^\infty e^{-\lambda t}\|T(t)\|\,dt
	\leq M\int_0^\infty e^{-(\lambda-\omega)t}\,dt
	=\frac{M}{\lambda-\omega}.
\end{equation}
These standard facts can be found, for example, in
\cite[Chapter~II]{EngelNagel2000}.

We shall use the Yosida approximants
\begin{equation}\label{eq:Yosida}
	J_\lambda:=\lambda R(\lambda,A),
	\qquad \lambda>\max\{\omega,0\}.
\end{equation}

\begin{lemma}[Yosida approximation]\label{lem:Yosida}
	For every $x\in X$,
	\[
	\lim_{\lambda\to\infty}J_\lambda x=x.
	\]
	Moreover, there is $\lambda_0>\max\{\omega,0\}$ such that
	\[
	\sup_{\lambda\geq\lambda_0}\|J_\lambda\|\leq 2M.
	\]
\end{lemma}

\begin{proof}
	Since $\int_0^\infty\lambda e^{-\lambda t}\,dt=1$, formula
	\eqref{eq:laplace-resolvent} gives
	\begin{equation}\label{eq:Yosida-difference}
		\begin{aligned}
			J_\lambda x-x
			&=\int_0^\infty\lambda e^{-\lambda t}T(t)x\,dt
			-\int_0^\infty\lambda e^{-\lambda t}x\,dt\\
			&=\int_0^\infty\lambda e^{-\lambda t}
			\bigl(T(t)x-x\bigr)\,dt.
		\end{aligned}
	\end{equation}
	Fix $\varepsilon>0$.  Strong continuity at zero provides $\delta>0$ such
	that
	\[
	\|T(t)x-x\|<\frac{\varepsilon}{2}
	\qquad(0\leq t\leq\delta).
	\]
	The part of \eqref{eq:Yosida-difference} over $[0,\delta]$ is therefore at
	most
	\[
	\frac{\varepsilon}{2}
	\int_0^\delta\lambda e^{-\lambda t}\,dt
	\leq\frac{\varepsilon}{2}.
	\]
	Using \eqref{eq:growth}, the remaining part is bounded by
	\begin{align*}
		&\int_\delta^\infty\lambda e^{-\lambda t}
		\bigl(\|T(t)x\|+\|x\|\bigr)\,dt\\
		&\quad\leq
		M\|x\|\lambda\int_\delta^\infty
		e^{-(\lambda-\omega)t}\,dt
		+\|x\|\lambda\int_\delta^\infty e^{-\lambda t}\,dt\\
		&\quad=
		\|x\|\left(
		M\frac{\lambda}{\lambda-\omega}
		e^{-(\lambda-\omega)\delta}
		+e^{-\lambda\delta}
		\right),
	\end{align*}
	which tends to zero as $\lambda\to\infty$.  This proves the strong
	convergence.
	
	Finally, by \eqref{eq:resolvent-bound},
	\begin{equation}\label{eq:Yosida-bound}
		\|J_\lambda\|
		\leq M\frac{\lambda}{\lambda-\omega}.
	\end{equation}
	If $\omega\leq0$, the last quotient is at most $1$ for every $\lambda>0$.
	If $\omega>0$, it is at most $2$ whenever $\lambda\geq2\omega$.
	Choosing, for instance,
	\[
	\lambda_0=\max\{1,2\max\{\omega,0\}\}
	\]
	proves the asserted estimate.
\end{proof}

Recall that a $C_0$-semigroup is uniformly continuous if and only if its
generator is bounded and everywhere defined
\cite[Theorem~I.3.7]{EngelNagel2000}.  Hence the Lotz property is equivalent
to the boundedness of every $C_0$-semigroup generator.

\subsection{Compact approximation properties}

\begin{definition}\label{def:CAP}
	A Banach space $X$ has the \emph{compact approximation property} (CAP) if,
	for every compact set $C\subset X$ and every $\varepsilon>0$, there is
	$K\in\Kop(X)$ such that
	\[
	\sup_{x\in C}\|Kx-x\|<\varepsilon.
	\]
	It has the \emph{bounded compact approximation property} (BCAP) if there is
	$C\geq1$ such that the compact approximants may always be chosen with
	$\|K\|\leq C$.  If the bound can be taken equal to $1$, then $X$ has the
	\emph{metric compact approximation property} (MCAP).
	
	Replacing compact operators by finite-rank operators gives the
	\emph{approximation property} (AP).  If the finite-rank approximants can be
	chosen with a common norm bound, one obtains the \emph{bounded approximation
		property} (BAP).
\end{definition}

The following elementary lemma will be used to pass from strong convergence
to uniform convergence on compact subsets.

\begin{lemma}\label{lem:strong-to-compact}
	Let $(L_\alpha)$ be a net in $\Lop(X)$ such that
	\[
	B:=\sup_\alpha\|L_\alpha\|<\infty
	\quad\text{and}\quad
	L_\alpha x\longrightarrow x
	\quad(x\in X).
	\]
	Then $L_\alpha\to\Id_X$ uniformly on each compact subset of $X$.
	Consequently, if every $L_\alpha$ is compact, then $X$ has the BCAP with
	constant at most $B$.
\end{lemma}

\begin{proof}
	Let $C\subset X$ be compact and let $\varepsilon>0$.  Put
	\[
	\eta:=\frac{\varepsilon}{3(B+1)}.
	\]
	Choose $x_1,\ldots,x_N\in C$ such that
	\[
	C\subset\bigcup_{j=1}^N B(x_j,\eta).
	\]
	For all sufficiently large $\alpha$,
	\[
	\|L_\alpha x_j-x_j\|<\frac{\varepsilon}{3}
	\qquad(j=1,\ldots,N).
	\]
	If $x\in C$ and $\|x-x_j\|<\eta$, then
	\begin{align*}
		\|L_\alpha x-x\|
		&\leq \|L_\alpha(x-x_j)\|
		+\|L_\alpha x_j-x_j\|+\|x_j-x\|\\
		&<B\eta+\frac{\varepsilon}{3}+\eta
		=\frac{\varepsilon}{3}+\frac{\varepsilon}{3}
		<\varepsilon.
	\end{align*}
	Thus the convergence is uniform on $C$.  If the $L_\alpha$ are compact,
	they are uniformly bounded compact approximants of the identity.
\end{proof}

Since finite-rank operators are compact, the approximation property implies
the CAP, and the BAP implies the BCAP.  The CAP need not imply the AP, even
when it holds metrically \cite{Willis1992}.  Our abstract criterion only
requires failure of the BCAP; the extraction theorem used later supplies the
stronger failure of the CAP.

\subsection{Strictly singular operators and Fredholm theory}

For the purposes of this paper, a linear subspace
$\mathcal J(X)\subset\Lop(X)$ is called a \emph{two-sided operator ideal on
	$X$} if
\[
ASB\in\mathcal J(X)
\qquad(A,B\in\Lop(X),\ S\in\mathcal J(X)).
\]

An operator $S\in\Lop(X)$ is strictly singular if no restriction of $S$ to
an infinite-dimensional closed subspace of $X$ is an isomorphism onto its
range.  The class $\SSop(X)$ is a closed two-sided operator ideal on $X$.
An operator $F$ is
\emph{Fredholm} if $\ker F$ is finite-dimensional, $\ran F$ is closed, and
$\ran F$ has finite codimension; its index is
\[
\ind F:=\dim\ker F-\codim\ran F.
\]
We use the standard Fredholm fact
\begin{equation}\label{eq:fredholm-fact}
	\alpha\Id_X+S\text{ is Fredholm of index }0
	\quad
	(\alpha\in\K\setminus\{0\},\ S\in\SSop(X)).
\end{equation}
This follows from the inessentiality of strictly singular operators; see, for
example, \cite[Chapter~5]{Pietsch1980} or
\cite[Section~IV.5]{Kato1995}.

The \emph{Calkin algebra} of $X$ is the unital Banach algebra
\begin{equation}\label{eq:Calkin}
	\Cal(X):=\Lop(X)/\Kop(X),
\end{equation}
and $\pi_X\colon\Lop(X)\to\Cal(X)$ denotes the quotient homomorphism.
Atkinson's theorem says that
\begin{equation}\label{eq:Atkinson}
	F\text{ is Fredholm}
	\quad\Longleftrightarrow\quad
	\pi_X(F)\text{ is invertible in }\Cal(X).
\end{equation}
We shall use only this equivalence; see, for example,
\cite[Section~IV.5]{Kato1995}.

If $A$ is the generator of a $C_0$-semigroup, then
\begin{equation}\label{eq:resolvent-range}
	R(\lambda,A):X\longrightarrow X
	\quad\text{is injective and}\quad
	\ran R(\lambda,A)=\Dom(A).
\end{equation}
If $A$ is unbounded, then $\Dom(A)\neq X$: otherwise the closed graph
theorem, applied to the closed operator $A:X\to X$, would imply that $A$ is
bounded.

\section{A Calkin-algebra criterion for the Lotz property}\label{SecCalkinAlg}

The next result isolates the operator-algebraic core of the argument.  An
element of a unital algebra is called \emph{noninvertible} if it has no
two-sided inverse.

\begin{theorem}[Calkin criterion]\label{thm:criterion}
	Let $X$ be a Banach space which fails the BCAP.  Suppose that there is
	$m\in\N$ such that
	\begin{equation}\label{eq:Calkin-nilpotence}
		a^m=0
		\qquad\text{for every noninvertible }a\in\Cal(X).
	\end{equation}
	Then $X$ has the Lotz property.
\end{theorem}

\begin{proof}
	Suppose that a $C_0$-semigroup $(T(t))_{t\geq0}$ on $X$ has an unbounded infinitesimal
	generator $A$.  Choose $M\geq1$ and $\omega\in\mathbb R$ that
	satisfy \eqref{eq:growth}, and let $\lambda>\max\{\omega,0\}$.  The resolvent
	$R(\lambda,A)$ is injective and has range $\Dom(A)$ by
	\eqref{eq:resolvent-range}.  If it were Fredholm, then its range would be
	closed.  Since $\Dom(A)$ is dense, this would give $\Dom(A)=X$, and the
	closed graph theorem would make $A$ bounded, a contradiction.  Hence
	$R(\lambda,A)$ is not Fredholm.
	
	By Atkinson's theorem, $\pi_X(R(\lambda,A))$ is noninvertible in
	$\Cal(X)$.  Assumption \eqref{eq:Calkin-nilpotence} therefore gives
	\[
	\pi_X(R(\lambda,A)^m)=\pi_X(R(\lambda,A))^m=0,
	\]
	so $R(\lambda,A)^m$ is compact.  The Yosida approximant
	$J_\lambda=\lambda R(\lambda,A)$ consequently satisfies
	\begin{equation}\label{eq:Yosida-power-compact}
		J_\lambda^m\in\Kop(X).
	\end{equation}
	
	Choose $\lambda_0$ as in \cref{lem:Yosida} and set
	\[
	C:=\sup_{\lambda\geq\lambda_0}\|J_\lambda\|\leq2M.
	\]
	Then $\|J_\lambda^m\|\leq C^m$.  Moreover,
	\begin{equation}\label{eq:power-factorization}
		J_\lambda^m-\Id_X
		=\bigl(\Id_X+J_\lambda+\cdots+J_\lambda^{m-1}\bigr)
		(J_\lambda-\Id_X).
	\end{equation}
	For every $x\in X$, the norm of the right-hand side is at most
	\[
	(1+C+\cdots+C^{m-1})\|(J_\lambda-\Id_X)x\|,
	\]
	which tends to zero by \cref{lem:Yosida}.  Thus the compact operators
	$J_\lambda^m$ are uniformly bounded and converge strongly to $\Id_X$.
	By \cref{lem:strong-to-compact}, $X$ has the BCAP, contradicting the
	hypothesis.  Therefore every $C_0$-semigroup generator on $X$ is bounded.
\end{proof}

The criterion has a particularly transparent few-operators consequence.

\begin{corollary}[Scalar-plus-strictly-singular criterion]
	\label{cor:few-operators-criterion}
	Suppose that
	\begin{equation}\label{eq:scalar-plus-SS}
		\Lop(X)=\K\Id_X+\SSop(X)
	\end{equation}
	and that, for some $m\in\N$,
	\begin{equation}\label{eq:SS-power-compact}
		S^m\in\Kop(X)\qquad(S\in\SSop(X)).
	\end{equation}
	If $X$ fails the BCAP, then $X$ has the Lotz property.
\end{corollary}

\begin{proof}
	Every $a\in\Cal(X)$ has the form
	\[
	a=\alpha 1_{\Cal(X)}+\pi_X(S)
	\qquad(\alpha\in\K,\ S\in\SSop(X)).
	\]
	By \eqref{eq:SS-power-compact}, $q:=\pi_X(S)$ satisfies $q^m=0$.  If
	$\alpha\neq0$, then $a$ is invertible, with inverse
	\[
	a^{-1}=\alpha^{-1}
	\sum_{j=0}^{m-1}(-\alpha^{-1}q)^j.
	\]
	Thus every noninvertible $a$ has $\alpha=0$ and hence $a^m=q^m=0$.
	Apply \cref{thm:criterion}.
\end{proof}

\begin{remark}\label{rem:products-versus-powers}
	The frequently occurring hypothesis that every product of $m$ strictly
	singular operators is compact implies \eqref{eq:SS-power-compact} by taking
	all factors equal.  The converse need not be assumed for the argument above.
	The Calkin formulation also shows exactly what is used: uniform nilpotence of
	the noninvertible part of the quotient algebra.
\end{remark}

\section{The Argyros--Motakis space and the extraction of a subspace}\label{SecArgyros-Motakis}

We recall the precise ingredients needed from the construction of Argyros
and Motakis.

If $Y$ and $Z$ are closed subspaces of a Banach space, then $Y+Z$ is called
a \emph{topological direct sum}, denoted by $Y\oplus Z$, if every element of
$Y+Z$ has a unique representation $y+z$, with $y\in Y$ and $z\in Z$, and
$Y+Z$ is closed.  Equivalently, the associated coordinate projections are
bounded.

A Banach space $X$ is \emph{hereditarily indecomposable} (HI) if no
infinite-dimensional closed subspace of $X$ is a topological direct sum of
two infinite-dimensional closed subspaces.  A sequence $(e_n)_{n\in\mathbb{N}}$ is a
\emph{basic sequence} if it is a Schauder basis for its closed linear span
\[
[e_n]:=\overline{\operatorname{span}}\{e_n:n\in\N\}.
\]
A bounded sequence $(x_n)$ in $X$ \emph{generates} a basic sequence $(e_n)$
as a spreading model if, for
every $k\in\N$ and every $(a_1,\ldots,a_k)\in[-1,1]^k$,
\[
\lim_{\substack{n_1\to\infty\\ n_1<\cdots<n_k}}
\left\|\sum_{j=1}^ka_jx_{n_j}\right\|
=\left\|\sum_{j=1}^ka_je_j\right\|,
\]
with uniform convergence in the coefficient cube $[-1,1]^k$.  Finally, a
sequence $(x_n)_{n\in\mathbb{N}}$ is \emph{weakly null} if $x^*(x_n)\to0$ for every
$x^*\in X^*$.

\begin{theorem}[Argyros--Motakis]\label{thm:AM}
	There is a separable reflexive real Banach space
	$\mathfrak X_{\mathrm{ISP}}$ with a Schauder basis such that:
	\begin{enumerate}[label=\textup{(\roman*)}]
		\item $\mathfrak X_{\mathrm{ISP}}$ is hereditarily indecomposable;
		\item every infinite-dimensional closed subspace
		$Y\subset\mathfrak X_{\mathrm{ISP}}$ admits a normalized weakly null
		sequence generating the unit vector basis of $\ell_1$ as a spreading
		model;
		\item for every infinite-dimensional closed subspace
		$Y\subset\mathfrak X_{\mathrm{ISP}}$ and every $T\in\Lop(Y)$, there
		are $\alpha\in\K$ and $S\in\SSop(Y)$ such that
		\[
		T=\alpha\Id_Y+S;
		\]
		\item if $Q,S,T\in\SSop(Y)$, then $QST\in\Kop(Y)$.
	\end{enumerate}
\end{theorem}

\begin{proof}[Source]
	These are items (i), (ii), (iii) and (v), respectively, of the main theorem
	of Argyros and Motakis \cite{ArgyrosMotakis2014}.  Their item (iii) is stated
	in the stronger form that every operator from $Y$ into
	$\mathfrak X_{\mathrm{ISP}}$ is a scalar multiple of the inclusion plus a
	strictly singular operator.  Restricting the codomain to $Y$ gives the form
	used here.
\end{proof}

We next explain rigorously why the spreading-model information gives access
to the Szankowski extraction theorem.

For $j\in\N$, the $j$th \emph{Rademacher function} on $[0,1]$ is
\[
r_j(t):=\operatorname{sgn}\bigl(\sin(2^j\pi t)\bigr),
\]
where we set $r_j(t)=1$ at the finitely many zeros of the sine.  Thus $r_j$
takes only the values $-1$ and $1$.

\begin{definition}
	Let $1\leq p\leq2$.  A Banach space $X$ has \emph{Rademacher type $p$} if
	there is a constant $T_p(X)<\infty$ such that, for every
	$x_1,\ldots,x_n\in X$,
	\begin{equation}\label{eq:typep}
		\left(
		\int_0^1
		\left\|\sum_{j=1}^n r_j(t)x_j\right\|^p\,dt
		\right)^{1/p}
		\leq T_p(X)
		\left(\sum_{j=1}^n\|x_j\|^p\right)^{1/p},
	\end{equation}
	where $(r_j)$ denotes the sequence of Rademacher functions.
	Type $p$ is called \emph{nontrivial} when $p>1$.
	By the Kahane--Khintchine inequalities, this formulation is equivalent to
	the customary definition with the second moment on the left-hand side.
\end{definition}

\begin{lemma}\label{lem:l1-no-type}
	If a Banach space $E$ admits a normalized sequence generating the unit vector
	basis of $\ell_1$ as a spreading model, then $E$ has no Rademacher type
	$p>1$.
\end{lemma}

\begin{proof}
	Let $(x_n)_{n\in\mathbb{N}}$ be such a normalized sequence.  From the definition of spreading
	model, for every $N\in\N$ one may choose indices
	\[
	k_1<k_2<\cdots<k_N
	\]
	sufficiently far out so that, simultaneously for every choice of signs
	$\varepsilon_1,\ldots,\varepsilon_N\in\{-1,1\}$,
	\begin{equation}\label{eq:l1-lower}
		\left\|\sum_{j=1}^N\varepsilon_jx_{k_j}\right\|
		\geq\frac12
		\left\|\sum_{j=1}^N\varepsilon_je_j\right\|_{\ell_1}
		=\frac{N}{2}.
	\end{equation}
	Indeed, for the fixed length $N$, the spreading-model convergence is uniform
	on the compact coefficient cube $[-1,1]^N$; it therefore applies in
	particular to the finite set $\{-1,1\}^N$.
	
	Assume that $E$ has type $p$ for some $p>1$.  Applying
	\eqref{eq:typep} to $x_{k_1},\ldots,x_{k_N}$, and using
	\eqref{eq:l1-lower} for the signs
	$\varepsilon_j=r_j(t)$, gives
	\begin{align*}
		\frac{N}{2}
		&\leq
		\left(
		\int_0^1
		\left\|\sum_{j=1}^N r_j(t)x_{k_j}\right\|^p\,dt
		\right)^{1/p}\\
		&\leq T_p(E)
		\left(\sum_{j=1}^N\|x_{k_j}\|^p\right)^{1/p}\\
		&=T_p(E)N^{1/p}.
	\end{align*}
	Thus
	\[
	N^{1-1/p}\leq2T_p(E)
	\qquad(N\in\N),
	\]
	which is impossible because $1-1/p>0$.  Hence $E$ has no type $p>1$.
\end{proof}

We use the following classical extraction consequence.

\begin{theorem}[Maurey--Pisier--Szankowski]\label{thm:MPS}
	Let $E$ be a separable Banach space.  If, for some $0<\varepsilon<1$, the
	space $E$ fails to have Rademacher type $2-\varepsilon$, then $E$ contains a
	closed infinite-dimensional subspace without the compact approximation
	property.
\end{theorem}

\begin{proof}[Reference]
	The local theory identifying the critical type is due to Maurey and Pisier
	\cite{MaureyPisier1976}, while the subspace construction is due to
	Szankowski \cite{Szankowski1978}.  The precise combined consequence, including
	failure of the compact approximation property, is recorded in
	\cite[Section~6, pp.~5--6]{DodosLopezAbadTodor2009}; see also
	\cite[Theorem~1.g.4]{LindenstraussTzafriri1979}.
\end{proof}

For brevity, we introduce terminology for the hereditary conclusion produced
by this extraction argument.

\begin{definition}\label{def:Lotz-saturated}
	A Banach space $E$ is \emph{Lotz-saturated} if every infinite-dimensional
	closed subspace of $E$ contains an infinite-dimensional closed subspace with
	the Lotz property.
\end{definition}

\begin{theorem}[Saturation criterion]\label{thm:saturation}
	Let $E$ be a separable Banach space.  Suppose that there is $m\in\N$ such
	that every infinite-dimensional closed subspace $Y\subset E$ satisfies:
	\begin{enumerate}[label=\textup{(\roman*)}]
		\item $Y$ contains a normalized sequence generating the unit vector basis
		of $\ell_1$ as a spreading model;
		\item $\Lop(Y)=\K\Id_Y+\SSop(Y)$;
		\item $S^m\in\Kop(Y)$ for every $S\in\SSop(Y)$.
	\end{enumerate}
	Then $E$ is Lotz-saturated.
\end{theorem}

\begin{proof}
	Let $Y\subset E$ be infinite-dimensional and closed.  By (i) and
	\cref{lem:l1-no-type}, $Y$ has no nontrivial Rademacher type.  The space $Y$
	is separable, so \cref{thm:MPS} gives an infinite-dimensional closed subspace
	$Z\subset Y$ without the CAP.  In particular, $Z$ fails the BCAP.  Applying
	(ii) and (iii) to $Z$, and then \cref{cor:few-operators-criterion}, shows
	that $Z$ has the Lotz property.  This proves the assertion for every $Y$.
\end{proof}

We first apply the saturation criterion to
$\mathfrak X_{\mathrm{ISP}}$.

\begin{corollary}\label{cor:ISP-saturated}
	The real space $\mathfrak X_{\mathrm{ISP}}$ is Lotz-saturated.  It is itself
	non-Lotz.
\end{corollary}

\begin{proof}
	The hereditary assertions in \cref{thm:AM} imply the hypotheses of
	\cref{thm:saturation} with $m=3$; for (iii), take the three factors in
	\cref{thm:AM}(iv) equal.  Hence the space is Lotz-saturated.  On the other
	hand, $\mathfrak X_{\mathrm{ISP}}$ has a Schauder basis, so
	\cref{prop:Schauder-non-Lotz} shows that it is non-Lotz.
\end{proof}

We can now obtain the announced real example.

\begin{theorem}\label{thm:main}
	There is a separable, reflexive, hereditarily indecomposable,
	infinite-dimensional real WCG Banach space $X_L$ such that
	\begin{enumerate}[label=\textup{(\roman*)}]
		\item $X_L$ does not have the compact approximation property;
		\item every $C_0$-semigroup on $X_L$ is uniformly continuous.
	\end{enumerate}
	In particular, $X_L$ has the Lotz property.
\end{theorem}

\begin{proof}
	The proof of \cref{thm:saturation}, applied to
	$Y=\mathfrak X_{\mathrm{ISP}}$, yields a closed infinite-dimensional subspace
	\begin{equation}\label{eq:XL}
		X_L\subset\mathfrak X_{\mathrm{ISP}}
	\end{equation}
	which fails the compact approximation property.
	
	Closed subspaces of separable reflexive spaces are separable and reflexive,
	so $X_L$ has these two properties.  By \cref{lem:separable-WCG}, it is WCG.
	Since
	$\mathfrak X_{\mathrm{ISP}}$ is hereditarily indecomposable, so is $X_L$.
	Finally, the hereditary assertions in \cref{thm:AM} give
	\[
	\Lop(X_L)=\K\Id_{X_L}+\SSop(X_L)
	\]
	and
	\[
	S^3\in\Kop(X_L)
	\qquad(S\in\SSop(X_L)).
	\]
	Applying \cref{cor:few-operators-criterion} with $m=3$ proves that $X_L$ has
	the Lotz property.
\end{proof}

The dual method of Argyros and Motakis supplies two further classes of
ambient spaces.  We use their complex versions; in a complex HI space every
operator on a closed subspace is scalar-plus-strictly-singular.

\begin{theorem}[Argyros--Motakis, dual method]
	\label{thm:AM-dual}
	The construction in \cite{ArgyrosMotakisDual} has the following properties.
	\begin{enumerate}[label=\textup{(\roman*)}]
		\item For every well-founded tree $\mathcal T$ admitted by the
		construction, there is a separable reflexive complex HI space
		$\mathfrak X_{\mathcal T}$ with a Schauder basis.
		\item There is a separable complex HI space
		$\mathfrak X_{\mathcal U}$ with a shrinking Schauder basis and no
		infinite-dimensional reflexive subspace.
		\item Every infinite-dimensional closed subspace of any of these spaces
		contains a normalized sequence generating the unit vector basis of
		$\ell_1$ as a spreading model.
		\item On every infinite-dimensional closed subspace $Y$ of any of these
		spaces, every bounded operator is scalar-plus-strictly-singular, and the
		product of any two strictly singular operators on $Y$ is compact.
	\end{enumerate}
\end{theorem}

\begin{proof}[Source]
	The structural and spreading-model assertions are the main theorems of
	\cite{ArgyrosMotakisDual}.  That paper also proves compactness of products of
	two strictly singular operators on every subspace.  The construction can be
	performed over the complex field; the scalar-plus-strictly-singular property
	on every closed subspace then follows from hereditary indecomposability, as
	recorded there.
\end{proof}

\begin{theorem}\label{thm:dual-Lotz-examples}
	Each complex space $\mathfrak X_{\mathcal T}$ and
	$\mathfrak X_{\mathcal U}$ in \cref{thm:AM-dual} is Lotz-saturated but is
	itself non-Lotz.  Consequently:
	\begin{enumerate}[label=\textup{(\roman*)}]
		\item every $\mathfrak X_{\mathcal T}$ contains a separable reflexive
		complex HI Lotz subspace;
		\item $\mathfrak X_{\mathcal U}$ contains a separable complex HI Lotz
		subspace with no infinite-dimensional reflexive subspace.
	\end{enumerate}
	In both cases the extracted Lotz subspace can be chosen to fail the CAP.
\end{theorem}

\begin{proof}
	Apply \cref{thm:saturation} with $m=2$, using
	\cref{thm:AM-dual}.  The ambient spaces have Schauder bases, so they are
	non-Lotz by \cref{prop:Schauder-non-Lotz}.  The extraction in the proof of
	\cref{thm:saturation} produces subspaces without the CAP.  Reflexivity passes
	to closed subspaces of $\mathfrak X_{\mathcal T}$, while every
	infinite-dimensional closed subspace of $\mathfrak X_{\mathcal U}$ remains nonreflexive by \cref{thm:AM-dual}(ii).  Separability and hereditary
	indecomposability also pass to closed subspaces, and separability implies the
	WCG property by \cref{lem:separable-WCG}.
\end{proof}

\section{The WCG question and further remarks}\label{SecWCG}

The two halves of the answer can now be placed side by side.

\begin{theorem}\label{thm:final-answer-WCG}
	The statement ``every infinite-dimensional WCG Banach space is non-Lotz'' is
	false.  It becomes true after adding the hypothesis of nonseparability.
	More precisely:
	\begin{enumerate}[label=\textup{(\roman*)}]
		\item if $X$ is a nonseparable WCG space, then $X$ is non-Lotz;
		\item the separable real space $X_L$ of \cref{thm:main} is both WCG and
		Lotz;
		\item the complex spaces extracted in \cref{thm:dual-Lotz-examples} are
		further separable WCG Lotz examples, including one with no
		infinite-dimensional reflexive subspace.
	\end{enumerate}
\end{theorem}

\begin{proof}
	Part (i) is \cref{cor:nonseparable-WCG-non-Lotz}.  For part (ii),
	\cref{thm:main} says that $X_L$ is separable and Lotz, while
	\cref{lem:separable-WCG} says that every separable Banach space is WCG.
	Part (iii) follows in the same way from
	\cref{thm:dual-Lotz-examples}.
\end{proof}

\begin{remark}[Questions left by the method]\label{rem:BCAP-Lotz}
	The argument leaves two natural questions: does there exist an
	infinite-dimensional separable Lotz space with the BCAP, and does there exist
	one with the CAP but without the BCAP?  The first question lies outside the
	present criterion, which uses failure of the BCAP to rule out an unbounded
	generator.  The second is compatible with the criterion, but the extraction
	theorem used here supplies the stronger failure of the CAP and would have to
	be replaced.  Finally, \cref{prop:Schauder-non-Lotz} implies immediately that
	none of the Lotz subspaces constructed here admits a Schauder decomposition.
\end{remark}

\begin{remark}[Lotz-saturated but non-Lotz]
	Each ambient Argyros--Motakis space used above has a Schauder basis and hence
	is non-Lotz by \cref{prop:Schauder-non-Lotz}.  Nevertheless,
	\cref{cor:ISP-saturated,thm:dual-Lotz-examples} show that every
	infinite-dimensional closed subspace contains a further Lotz subspace.  Thus
	the Lotz property itself is not hereditary upward from subspaces, and the
	passage to a subspace without the CAP is essential to these constructions.
\end{remark}

\begin{remark}[An integral-free Yosida calculation]
	There is a shorter alternative proof of the strong convergence in
	\cref{lem:Yosida}.  If $x\in\Dom(A)$, the resolvent identity gives
	\begin{equation}\label{eq:Yosida-algebraic}
		\lambda R(\lambda,A)x-x
		=R(\lambda,A)Ax.
	\end{equation}
	Indeed, $(\lambda\Id-A)x=\lambda x-Ax$, and applying
	$R(\lambda,A)$ gives
	$x=\lambda R(\lambda,A)x-R(\lambda,A)Ax$.  By
	\eqref{eq:resolvent-bound}, for $\lambda\geq\lambda_0$,
	\[
	\|R(\lambda,A)Ax\|
	\leq\frac{M}{\lambda-\omega}\|Ax\|\longrightarrow0.
	\]
	Thus $J_\lambda x\to x$ on the dense set $\Dom(A)$.  The uniform estimate
	$\sup_{\lambda\geq\lambda_0}\|J_\lambda\|\leq2M$ then extends the
	convergence to every $x\in X$: given $y\in\Dom(A)$,
	\[
	\|J_\lambda x-x\|
	\leq(2M+1)\|x-y\|+\|J_\lambda y-y\|.
	\]
	First choose $y$ close to $x$ and then let $\lambda\to\infty$.
\end{remark}

\begin{remark}[What the Calkin proof produces]
	If an unbounded generator existed on a space satisfying the nilpotence
	hypothesis of \cref{thm:criterion}, then the family
	\[
	\bigl\{[\lambda R(\lambda,A)]^m:\lambda\geq\lambda_0\bigr\}
	\]
	would be a uniformly bounded family of compact operators converging strongly,
	and hence uniformly on compact subsets, to the identity.  Thus the proof
	establishes the sharper implication
	\[
	\begin{gathered}
		\text{existence of an unbounded $C_0$-semigroup generator}
		\\
		\Longrightarrow
		\text{bounded compact approximation property}.
	\end{gathered}
	\]
	Thus failure of the BCAP is the exact approximation-theoretic obstruction
	used by the argument; failure of the CAP is a convenient stronger condition
	provided by the extraction theorem.
\end{remark}

\begin{remark}[Scalar field]
	The Calkin and saturation criteria are valid over both $\mathbb R$ and
	$\mathbb C$.  The $\mathfrak X_{\mathrm{ISP}}$ application is stated over
	the real field exactly as in \cite{ArgyrosMotakis2014}.  For the dual-method
	spaces we deliberately use the complex versions, because hereditary
	indecomposability over $\mathbb C$ yields the scalar-plus-strictly-singular
	property on every closed subspace.
\end{remark}

\begin{remark}[Bibliographic context]
	The examples emphasized in Lotz's theorem and in the subsequent literature
	are nonseparable in infinite dimension; the model class is formed by
	$L^\infty$-type spaces, which are Grothendieck and have the Dunford--Pettis
	property \cite{Lotz1985,Budde2024}.  As explained in the introduction, a
	separable space with both of these properties must be finite-dimensional.
	The spaces constructed here are instead separable and infinite-dimensional.
	The author is not aware of a previously recorded example of this kind.  This
	is only a bibliographic qualification: the proofs combine established
	ingredients, and no priority claim is needed for the mathematical results.
\end{remark}

\begin{remark}[Further quotient-algebra directions]
	The proof suggests looking beyond scalar-plus-strictly-singular spaces.  Any
	Banach space failing the BCAP whose Calkin algebra has uniformly nilpotent
	noninvertible elements satisfies the Lotz property, whether or not its Calkin
	algebra is generated by the identity and images of strictly singular
	operators.  Recent work on nilpotent quotients of strictly singular operators
	on direct sums provides related algebraic phenomena
	\cite{LaustsenWirzenius2025}, although spaces with Schauder bases are
	automatically non-Lotz and therefore cannot be fed directly into
	\cref{thm:criterion}.
\end{remark}

\section*{Declarations}

\noindent\textbf{Competing interests.}
The authors declare that they have no competing interests.

\medskip

\noindent\textbf{Generative AI and AI-assisted technologies.}
During the preparation of this manuscript, the authors used OpenAI's ChatGPT as a generative-AI tool for exploratory calculations, consistency checks, organization of arguments, and drafting and editorial assistance. All AI-assisted material retained in the manuscript was independently reviewed and verified by the authors. The manuscript was written and edited by the authors, who take full responsibility for the content of the work.

\end{document}